\documentclass[letterpaper, 10 pt, conference]{ieeeconf} 
\IEEEoverridecommandlockouts                              % This command is only needed if                                                           % you want to use the \thanks comman
\usepackage[utf8]{inputenc}  
  
\usepackage{mathtools}
\mathtoolsset{showonlyrefs,showmanualtags}
\usepackage{amssymb,amsthm}
\usepackage{dsfont}
\usepackage{booktabs} 
\usepackage{color}
\usepackage{bbold} 
\usepackage[yyyymmdd,hhmmss]{datetime}
\usepackage{bm,upgreek}
\usepackage{tablefootnote}
\usepackage{nth}
\usepackage{pgfplots}
\usepackage{tikz}
\usetikzlibrary{shapes,snakes}
\usetikzlibrary{patterns}
\usepackage{multirow}
\usepackage{footnote}

\usepackage{nth}

\usepackage[ruled]{algorithm} 
\usepackage{algpseudocode}

\newtheorem{theorem}{Theorem}

\newtheorem{lemma}{Lemma}

\newcommand{\Exp}[1]{\mathbb{E}[ #1]}
\newcommand{\Prob}[1]{\mathbb{P}(#1)} 
\newcommand{\ID}[1]{ \mathbb{1} (#1 ) }
\newcommand{\B}{ \mathtt{blank} }

\newcommand{\TM}{T_{\scriptstyle \mathsf{m}}}

\DeclareMathOperator{\Binopdf}{binopdf}

\usepackage{xcolor}

\begin{document}

%\title{Operation Management in Fault-Tolerant Systems with Homogeneous Components under Uncertainty}
\title{Near-Optimal Design for Fault-Tolerant Systems with Homogeneous Components under Incomplete Information}
\author{Jalal Arabneydi and Amir G. Aghdam% <-this % stops a space
\thanks{ This work has been supported in part by the Natural Sciences and Engineering Research Council of Canada (NSERC) under Grant RGPIN-262127-17, and in part by Concordia University under Horizon Postdoctoral Fellowship.}  
\thanks{Jalal Arabneydi and Amir G. Aghdam are with the  Department of Electrical and Computer Engineering, 
        Concordia University, 1455 de Maisonneuve Blvd. West, Montreal, QC, Canada, Postal Code: H3G 1M8.  Email:{\small jalal.arabneydi@mail.mcgill.ca},
          and 
        {\small aghdam@ece.concordia.ca.}}%
}

% The paper headers
%\markboth{Journal of \LaTeX\ Class Files,~Vol.~11, No.~4, December~2012}%
\markboth{}%
{Arabneydi \MakeLowercase{\textit{et al.}}: }

\maketitle

\vspace*{-5cm}{\footnotesize{Proceedings of IEEE  International Midwest Symposium on Circuits and Systems, 2018.}}
\vspace*{3.85cm}

 \begin{abstract}
In this paper, we study a fault-tolerant control for systems consisting of multiple homogeneous components such as parallel processing machines. This type of system is often more robust to uncertainty compared to those with a single component.  The state of each component is either in the operating mode or faulty.   At any time instant, each component may independently become faulty according to a Bernoulli probability distribution. If a component is faulty, it remains so  until it is fixed.  The objective is to design a fault-tolerant system by sequentially choosing one of the following three options: (a)  do nothing  at zero cost; b)  detect the number of faulty components at the  cost of inspection, and c) fix the system  at  the  cost of repairing faulty components.   A Bellman equation is developed  to identify  a near-optimal solution for the problem. The efficacy of the proposed solution is verified by numerical simulations.
\end{abstract}
\section{Introduction}

In the design of control systems for industrial applications, it is important to achieve a certain level of fault tolerance. There has been a growing interest in the literature recently on developing effective fault tolerant paradigms for reliable control of real-world systems \cite{Frank2004}. This type of control system is particularly useful when the system is subject to unpredictable failures.  Recent applications of fault-tolerant control include power systems and aircraft flight control systems \cite{odgaard2014fault,Verhaegen2010}.

It is known that certain class of faults can be modeled as partially observable Markov decision processes (POMDP). Various methods are studied in the literature to find an approximate solution to POMDP. Grid-based methods are used in \cite{lovejoy1991computationally} to compute an approximate value function at a fixed number of points in the belief space, and then interpolate over the entire space. The advantage of such approaches is that their computational complexity remains unchanged at each iteration and does not increase with time. Their drawback, however, is that the fixed points may not be reachable. In point-based methods \cite{Shani2013survey}, the reachability drawback is circumvented by restricting attention to the reachable set. Using the notion of $\alpha$-vectors, an approximate value function is obtained iteratively over a finite number of points in the reachable set. In these methods, the points are not fixed and may change with the value function. In policy-search methods such as finite-state controllers \cite{hansen1998finite}, on the other hand, attention is devoted to a certain class of strategies, and the objective is to find the best strategy in that class using policy iteration and gradient-based techniques. For more details on POMDP solvers, the interested reader is referred to \cite{Cassandra1998survey, Shani2013survey,Murphy2000survey}, and references therein.

In this paper, a fault-tolerant scheme is proposed for a system consisting of a number of homogeneous components, where each component may fail with a certain probability. Three courses of action are defined to troubleshoot the faulty system: (i) let the system operate with faulty components; (ii) inspect the system, and (iii) repair the system. Each course of action has an implementation cost. The problem is formulated as a POMDP but since finding an optimal solution for this problem is intractable, in general, we are interested in seeking a near-optimal solution for it~\cite{madani2003undecidability}. However, identifying an $\epsilon$-optimal solution for this problem is also NP-hard \cite{meuleau1999solving}. To overcome this hurdle, we exploit the structure of the problem to use a different information state (that is smaller than the belief state). The computational complexity of the proposed solution is logarithmic with respect to the desired neighborhood $\epsilon$, and polynomial with respect to the number of components. To derive some of the results, we use some methods developed in \cite{JalalACC2018,arabneydi2016new}.

This paper is organized as follows. The problem is formally stated in Section~\ref{sec:problem} and the main results of the work are presented in the form of three theorems in Section~\ref{sec:main}. Two numerical examples are provided in Section~\ref{sec:examples}  and some concluding remarks are given in Section~\ref{sec:conclusions}.

\section{Problem Formulation}\label{sec:problem}
Throughout this paper, $\mathbb{R}$ and $\mathbb{N}$  refer, respectively,  to  real and natural numbers. For any $n \in \mathbb{N}$,  $\mathbb{N}_n$  denotes the finite set $\{1,\ldots,n\}$. Moreover,  $\Prob{\boldsymbol \cdot}$   is the probability of a random variable,  $\Exp{\boldsymbol \cdot}$ is the expectation of a random variable, and $\ID{\boldsymbol \cdot }$  is the indicator function. The shorthand notation $x_{a:b}$  denotes vector $(x_a,\ldots,x_b)$, $a,b \in \mathbb{N}, a \leq b$. For any finite set $\mathcal{X}$, $\mathcal{P}(\mathcal{X})$ denotes the space of probability measures on $\mathcal{X}$. For any $y,n \in \mathbb{N}$, $y \leq n$,  and $p \in [0,1]$,  $\Binopdf(y,n,p)$ is the  Binomial probability distribution function of  $y$ successful outcomes from  $n$ trials where the success  probability is~$p$. 

Consider a stochastic dynamic system consisting of $n \in \mathbb{N}$  internal components. Denote by  $x^i_t \in \mathcal{X}:=\{0,1\}$  the state of component $i \in \mathbb{N}_n$ at time $t \in \mathbb{N}$, where $x^i_t=0$ means  that  the $i$-th component is in  the operating mode and $x^i_t=1$ means that  it is faulty. If a component is faulty, it  remains so until it is  repaired.  Let $p \in [0,1]$  denote the probability that a component becomes faulty at any time $t \in \mathbb{N}$. It is assumed that  the probability of   failure of each component is independent of  others.  

 Denote by $m_t=\sum_{i=1}^n x^i_t \in \mathcal{M}:= \{0,1,\ldots,n\}$ the number of faulty components at time $t \in \mathbb{N}$, and note that the state of each component may not be directly available.  Let $o_t \in \mathcal{O}= \mathcal{M} \cup \{\B\}$ denote the number of faulty components at  time $t \in \mathbb{N}$ that are observed. If there is no observation at time $t \in \mathbb{N}$,  then $o_t=\B$.  Initially,    the system  is assumed to have  no faulty components, i.e.  $o_1=0$. 
 
At  any time $t \in \mathbb{N}$, we have  three different options (actions) at our disposal, represented by $u_t \in \mathcal{U}$, where $\mathcal{U}=\mathbb{N}_3$ is the action set.  The first option is to do nothing and let the system continue operating without disruption at no implementation cost.   In this case,   no new  information on the  number  of faulty components is collected, i.e.,
\begin{equation}\label{eq:o_1}
\Prob{o_{t+1}=\B | m_t=m, u_t=1}=1, \quad m \in \mathcal{M},t \in \mathbb{N}.
\end{equation}
The second option is to  inspect the system and  detect the number of faulty components at some inspection cost, where 
\begin{equation}\label{eq:o_2}
\Prob{o_{t+1}=m | m_{t+1}=m, u_t=2}=1, \quad m \in \mathcal{M},  t \in \mathbb{N}.
\end{equation}
 The third option  is to repair the faulty components   at a cost depending on the number of  them, i.e. $m_t$.   Therefore,   at any time $t \in \mathbb{N}$,  the  following relations hold:
\begin{align}\label{eq:relation_prob}
\Prob{x^i_{t+1}=1 | x^i_t=0, u_{t} \in \{ 1,2\}}&=p, \nonumber \\
\Prob{x^i_{t+1}=1 | x^i_t=1, u_{t} \in \{ 1,2\}}&=1, \nonumber \\
%\Prob{o_{t+1}=0 | m_{t}=m, u_t=3}&=1, \nonumber \\
\Prob{x^i_{t+1}=1 | x^i_t=0, u_{t}=3}&=0, \nonumber \\
\Prob{x^i_{t+1}=1 | x^i_t=1, u_{t}=3}&=0.
\end{align}

Let $c(m,u): \mathcal{M}\times \mathcal{U} \rightarrow \mathbb{R}$ denote the cost associated with action $u \in \mathcal{U}$ when the number of faulty  components is $m \in \mathcal{M}$. The strategy $g$  is defined as  the mapping from the available information by time  $t$ to an action in $\mathcal{U}$, i.e.,
\begin{equation}
u_t=g(o_{1:t},u_{1:t-1}).
\end{equation}

The objective is  to  develop a cost-efficient fault-tolerant strategy in  the sense that the system operates with a  relatively small number of faulty components, taking the inspection and repair costs into account. To this end,  given the discount factor $\beta \in (0,1)$, we define the following  cost:
\begin{equation}\label{eq:total_cost}
J(g)=\Exp{\sum_{t=1}^\infty \beta^{t-1} c(m_t,u_t)}.
\end{equation}
\section{Main Result}\label{sec:main}
To present the main result of this paper, we first  derive a Bellman equation  to identify  the optimal solution. Since  the corresponding   Bellman equation  involves  an intractable optimization problem,  we  subsequently present an alternative  Bellman equation that is tractable and  provides a near-optimal solution. For any  $m \in \mathcal{M} \backslash \{0,n\}$, define  the following  vector-valued function $\phi: \mathcal{M} \backslash \{0,n\} \rightarrow \mathcal{P}(0,1,.\ldots,n)$:
\begin{equation}
\phi(m, \boldsymbol \cdot):=\Binopdf(\boldsymbol \cdot, m,1) \ast \Binopdf(\boldsymbol \cdot, n-m,p),
\end{equation}
where  $\ast$ denotes the convolution operator.
\begin{theorem}\label{thm:1}
 Given any realization $m_{1:t}$ and $u_{1:t}$, $t \in \mathbb{N}$, the transition probability matrix of the number of faulty components can be computed as follows:
 \begin{equation}
 \Prob{m_{t+1}|m_{1:t},u_{1:t}}= \Prob{m_{t+1}|m_{t},u_{t}},
 \end{equation}
 where for any $y \in \{0,1,\ldots,n\}$,
\begin{itemize}
\item if  $m_t=0$ and $u_t \in \{1,2\}$, then 
\[\Prob{m_{t+1}=y|m_t,u_t}= \Binopdf(y,n,p),\]
\item if  $m_t=n$ and $u_t \in \{1,2\}$, then 
\[\Prob{m_{t+1}=y|m_t,u_t}= \Binopdf(y,n,1)= \ID{y=n},\]
\item if  $m_t \notin \{0,n\}$ and $u_t \in \{1,2\}$, then 
\[\Prob{m_{t+1}=y|m_t,u_t }= \phi(m_t, y+1),\]
\item if   $u_t=3$, then 
\[\Prob{m_{t+1}=y|m_t,u_t }= \ID{y=0}.\]
\end{itemize}
\end{theorem}
\begin{proof}
 Define $\hat x^i_{t+1}:= x^i_t x^i_{t+1}$ and  $\tilde x^i_{t+1}:= (1-x^i_t) x^i_{t+1}$, $i \in \mathbb{N}_n, t \in \mathbb{N}$.  Define also $\mathbf x_t:=(x^1_t,\ldots,x^n_t)$, $t \in \mathbb{N}$. Then, given any realization $\mathbf x_{1:t}$ and $u_{1:t}$,  one has
\begin{align}\label{eq:proof_thm1}
\Prob{m_{t+1} | \mathbf x_{1:t}, u_{1:t}}&=\Prob{\sum_{i=1}^n x^i_{t+1} | \mathbf x_{t}, u_{t}} \nonumber \\
&= \Prob{\sum_{i=1}^n \hat x^i_{t+1}  + \sum_{i=1}^n \tilde x^i_{t+1}| \mathbf x_{t}, u_{t}}.
\end{align}
On the other hand, one can conclude  from the above definitions that  $(n-m_t)$ terms of $ \sum_{i=1}^n \hat x^i_{t+1}$  as well as $m_t$ terms of $\sum_{i=1}^n \tilde x^i_{t+1}$ are definitely zero. It is also important to note that $\hat x^i_{t+1}$ and $\tilde x^i_{t+1}$,  $t \in \mathbb{N}, i \in \mathbb{N}_n$, are independent Bernoulli random variables with success probability $x^i_t \cdot \Prob{x^i_{t+1}=1|x^i_t=1,u_t}$ and  $(1-x^i_t) \cdot \Prob{x^i_{t+1}=1|x^i_t=0,u_t}$, respectively. Therefore, the right-hand side of~\eqref{eq:proof_thm1} is the probability of the sum  of $m_t$  i.i.d.  Bernoulli random variables with  success probability $\Prob{x^i_{t+1}=1|x^i_t=1,u_t}$ and  $(n-m_t)$ i.i.d.  Bernoulli random variables with  success probability $\Prob{x^i_{t+1}=1|x^i_t=0,u_t}$.  Since the random variables are independent, the probability of their sum is equal to  the convolution of their probabilities. The proof follows from~\eqref{eq:relation_prob}, on noting that $m_{1:t}$ can be represented by $\mathbf{x}_{1:t}$.
\end{proof}

Let $s_t \in \mathcal{M}$ be the last observation before $t \in \mathbb{N}$ that is not \emph{blank} and $z_t \in \mathbb{N} \cup \{0\}$ be the elapsed time associated with it, i.e., the time interval between  the observation of $s_t$ and~$t$.
\begin{lemma}\label{lemma:hat-f}
There exists a function $\hat f$ such that
\begin{equation}
(s_{t+1},z_{t+1})=\hat f(s_t,z_t,u_t,o_{t+1}), \quad t \in \mathbb{N},
\end{equation} 
where 
\begin{equation}
\hat f(s_t,z_t,u_t,o_{t+1}):=\begin{cases}
(s_t, 1+z_t), & u_t=1,\\
(o_{t+1},0), & u_t=2,\\
(0,0), & u_t=3.
\end{cases}
\end{equation}
\end{lemma}
\begin{proof}
The proof follows from the definition of $s_t$ and $z_t$, and equations~\eqref{eq:o_1},~\eqref{eq:o_2} and~\eqref{eq:relation_prob}.
\end{proof}
 For the sake of  simplicity, denote by $\TM$ the transition probability matrix of the number of faulty components under actions  $\{1,2\}$ given by Theorem~\ref{thm:1}, i.e., for any $m',m \in \mathcal{M}$ and  $t \in \mathbb{N}$,
\begin{equation} \label{eq:tm}
\TM(m',m):=\Prob{m_{t+1}=m' | m_t=m}.
\end{equation} 

\begin{lemma}\label{lemma:prob_o}
Given any realization $s_{1:t}$, $z_{1;t}$, $o_{1:t}$ and $u_{1:t}$, $t \in \mathbb{N}$, the following equality  holds irrespective of strategy~$g$,
\begin{multline}\label{eq:prob_o}
\Prob{o_{t+1} | s_{1:t}, z_{1:t}, o_{1:t},u_{1:t}}= \ID{u_t=1, o_{t+1}=\B}\\
\ID{u_t=2} \TM^{1+z_t}(o_{t+1},s_t)+ \ID{u_t=3,o_{t+1}=0}.
\end{multline}
\end{lemma}
\begin{proof}
The proof follows from  equations~\eqref{eq:o_1},~\eqref{eq:o_2},~\eqref{eq:relation_prob} and~\eqref{eq:tm}, and the Chapman–Kolmogorov equation.
\end{proof}

\begin{lemma}
Given any realization $s_{1:t}$, $z_{1;t}$, $o_{1:t}$ and $u_{1:t}$, $t \in \mathbb{N}$, there exists a function $\hat c$ such that
\begin{equation}
\Exp{c(m_t,u_t)| s_{1:t},z_{1:t},o_{1:t}, u_{1:t}}=\hat c(s_t,z_t,u_t),
\end{equation}
where
\begin{equation}\label{eq:hat_c}
\hat c(s_t,z_t,u_t):=\sum_{m \in \mathcal{M}} c(m,u_t) \TM^{z_t}(m,s_t).
\end{equation}
\end{lemma}
\begin{proof}
The proof follows from the definition of expectation operator,  states $(s_t,z_t)$,  update function $\hat f$ in Lemma~\ref{lemma:hat-f}, and  the Chapman–Kolmogorov equation.
\end{proof}

\begin{theorem}
For any $ s\in \mathcal{M}$ and $z \in \mathbb{N} \cup \{0\}$, define the following Bellman equation
\begin{equation}\label{eq:bellman-optimal}
V(s,z)=\min_{u \in \mathcal{U}} (\hat c(s,z,u) + \beta \Exp{V(\hat f(s,z,u,o))} ),
\end{equation}
where   the expectation is taken over observations $ o \in \mathcal{O}$ with respect to the conditional probability function in~\eqref{eq:prob_o}. The optimal strategy for the cost function~\eqref{eq:total_cost} is obtained by solving the above equation.
\end{theorem}
\begin{proof}
The proof follows from the fact that $(s_t,z_t)$, $t \in \mathbb{N}$,  is an information state because it evolves in a Markovian manner under control action $u_t$ according to Lemma~\ref{lemma:hat-f}. In addition, the conditional  probability~\eqref{eq:prob_o} and expected cost~\eqref{eq:hat_c} do not depend on  strategy $g$, and can be represented in terms of state $(s_t,z_t)$ and action $u_t$, $t \in \mathbb{N}$. Thus, the proof is completed by using the  standard results from Markov decision theory~\cite{Bertsekas2012book}.
\end{proof}
Since the optimization of  the  Bellman equation~\eqref{eq:bellman-optimal} is carried out over a countable infinite set, it is computationally difficult to   solve it. As a result, we  are interested in a strategy which is sufficiently close to the optimal strategy and is tractable. To this end,  define the  following Bellman equation for any $k \in \mathbb{N}$, $s \in \mathcal{M}$ and $z \in \{0,1,\ldots,k\}$:
\begin{equation}\label{eq:bellman_k}
\hat V_k(s,z)=\min_{u \in \mathcal{U}}(\hat V^u_k(s,z)),
\end{equation}
where
\begin{align}
\hat V^1_k(s,z) \hspace{-.1cm}&:= \hspace{-.1cm} \hat c(s,z,1)  \hspace{-.1cm}+\hspace{-.1cm}  \beta \hspace{-.05cm}\left( \hspace{-.05cm} \hat V_k(s,1+z) \ID{z \hspace{-.1cm}<\hspace{-.1cm} k} \hspace{-.1cm}+ \hspace{-.1cm} \hat V_k(0,0) \ID{z=k} \hspace{-.1cm} \right) \nonumber \\
\hat V^2_k(s,z)&:= \hat c(s,z,2) + \beta \sum_{m \in \mathcal{M}}     \TM^{1+z}(m,s)    \hat V_k(m,0), \nonumber \\
\hat V^3_k(s,z)&:= \hat c(s,z,3) + \beta \hat V_k(0,0).
\end{align}

Let $c_{\text{max}}$ denote an upper bound on the per-step cost and $J^\ast $ denote the cost under the optimal strategy.
\begin{theorem}\label{thm:app}
Given $\varepsilon \in \mathbb{R}_{>0}$, choose  a sufficiently large $k \in \mathbb{N}$ such that
%\begin{equation}
$k \geq   \log\left( \frac{(1-\beta)\varepsilon}{2c_{\text{max}}}\right) / \log{\beta}.$
%\end{equation}
An  $\varepsilon$-optimal strategy can then  be obtained by  solving the  Bellman equation~\eqref{eq:bellman_k} as follows:
\begin{equation}\label{eq:proposed-strategy}
g^\ast_\varepsilon(s,z):=\begin{cases}
1, & \hat V^1_k(s,z) = \hat V_k(s,z),\\ 
2, & \hat V^2_k(s,z) =\hat  V_k(s,z),\\ 
3, & \hat V^3_k(s,z) = \hat V_k(s,z),
\end{cases}
\end{equation}
where $|J^\ast - J(g^\ast_\varepsilon)| \leq \varepsilon$.
\end{theorem}
\begin{proof}
Due to space limitations, only a sketch of the proof is provided, which consists of two steps. In  the first step, an approximate Markov decision process with state space $\mathcal{M} \times (0,1,\ldots,k)$ and action space $\mathcal U$ is constructed in such a way that it complies with the dynamics and cost of the original model. The optimal solution of the  approximate model is  obtained from the Bellman equation~\eqref{eq:bellman_k}. In the second step, it is shown that the difference between the optimal cost of the original  model  and that of the  approximate  model  is upper-bounded by $ \frac{2 \beta^k c_{\text{max}}}{1-\beta}$. If $\varepsilon \leq  \frac{2 \beta^k c_{\text{max}}}{1-\beta}$, then the solution of the approximate  model is an $\varepsilon$-optimal solution for the original model.
\end{proof}

According  to the strategy proposed in Theorem~\ref{thm:app} (described by~\eqref{eq:proposed-strategy}),  the near-optimal action at  any time $t \in \mathbb{N}$  depends on the latest observation of the number of faulty components by that time ($s_t$) and the elapsed time since then $(z_t)$.  Note that the near-optimal action  changes sequentially  in time based on the dynamics of the state $(s_t,z_t)$, according to~Lemma~\ref{lemma:hat-f}.
\section{Simulations}\label{sec:examples}
In this section,   we  aim to verify  the main result  presented in the preceding section by simulations. Consider a computing platform consisting of $n$ processors. Let $m_t$ denote the number of faulty processors at time $t \in \mathbb{N}$ and $p$ be  the probability that a processor fails.   The per-step cost under action $u_t=1$ is expressed  as follows:
%\begin{equation}
$c(m_t, 1)= \ell(m_t),$
%\end{equation}
where $\ell(m_t) \in \mathbb{R}_{\geq 0}$  is  the cost of  operating with  $m_t$ faulty processors.
%,  e.g., $\ell(m_t)= \ID{m_t \geq m_{min}}$, $m_{min} \in \mathcal{M}$, means there is no cost  if  the number of faulty processors is less than threshold $m_{min}$. 
The per-step cost under action $u_t=2$ is described as:
%\begin{equation}
$c(m_t, 2)= \ell(m_t) + \alpha(m_t),$
%\end{equation}
where $\alpha(m_t) \in \mathbb{R}_{\geq 0}$ is the cost of inspecting the system  to detect the number of faulty processors. The per-step cost under action $u_t=3$ is formulated  as:
%\begin{equation}
$c(m_t, 3)= \ell(m_t) + \gamma(m_t),$
%\end{equation}
where $\gamma(m_t) \in \mathbb{R}_{\geq 0}$ is the cost of repairing the faulty processors.
\\

 \textbf{Example 1. Inspection and repair with fixed price.} Let the cost of  inspection and repair be constant, i.e.,  they do not depend on the number of faulty components. We consider the following numerical parameters:
\begin{multline}
n=10, \quad p=0.02, \quad\beta=0.95, \quad  k=100,\\
\quad \ell(m)=m, \quad \alpha (m)=1, \quad \gamma(m)=60, \quad \forall m \in \mathcal{M}.
\end{multline}

 Figure~\ref{fig:flat}  shows the optimal course of action for the above setting, in different scenarios in terms of the number of faulty processors (based on the most recent observation). In this figure, the black color represents the first option  (continue operating without disruption), gray color  represents the second option (inspect the system and detect the number of faulty components) and the white color represents the third option (repair the faulty components). It is observed from the figure that  that the inspection and repair options  become more attractive as the number of faulty processors   and/or the elapsed time since the last observation  grow.
\begin{figure}
\center
\vspace{-0cm}
\scalebox{0.94}{
\includegraphics[width=\linewidth]{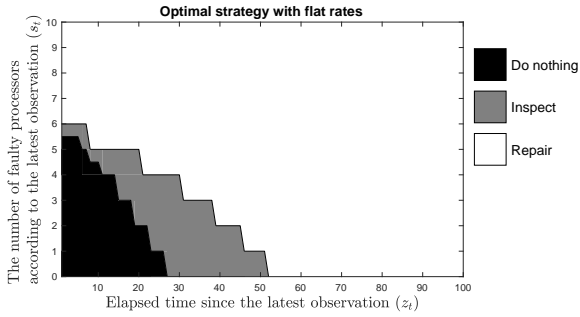}}
\vspace{-0cm}
\caption{The optimal strategy for Example 1,  where  the options are depicted in three different colors.}\label{fig:flat}
\end{figure}
\\

 \textbf{Example 2. Inspection and repair with variable price.} Now, let the cost of  inspection and repair be  variable.  We consider the same parameters as the previous  example, except the following ones:
\begin{equation}
\alpha (m)=0.1+0.09m,  \gamma(m)=30+3m, \quad  \forall m \in \mathcal{M}.
\end{equation}
The results are presented in Figure~\ref{fig:variable}, analogously to Figure~\ref{fig:flat}. The figure shows that   the  inspection option  is less desirable compared to Example 1, where the inspection and repair options prices were independent of the number of faulty processors. The reason is that  with the variable rate, the repair option becomes  more economical, hence more attractive than the previous case.
\begin{figure}[t!]
\center
\vspace{0cm}
\scalebox{0.94}{
\includegraphics[width=\linewidth]{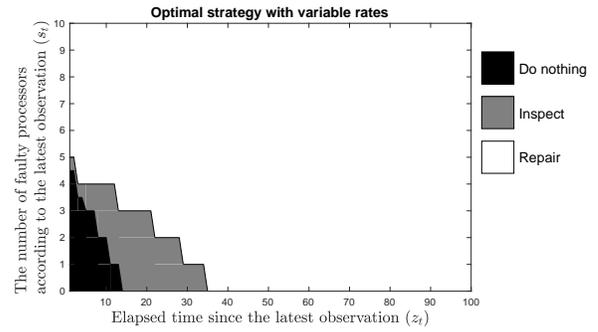}}
\vspace{0cm}
\caption{The optimal strategy for Example 2,  where the options are depicted in three different colors.}\label{fig:variable}
\end{figure}
\section{Conclusions}\label{sec:conclusions}
In this paper, we presented a fault-tolerant scheme for a system consisting of a number of homogeneous components, where each component can  fail at any time with a prescribed  probability. We proposed a near-optimal strategy to choose sequentially between three options: (1) do nothing and let the system operate with faulty components; (2) inspect to detect the number of faulty components, and (3) repair the faulty components. Each option incurs a cost that is incorporated in the overal cost function in the optimization problem. Two numerical examples are presented to demonstrate the results in the cases of fixed and variable rates.  As a future work, one can investigate the case where there are a sufficiently large number of components using  the  law of large numbers~\cite{JalalCDC2017}. 
%\printbibliography
\bibliographystyle{IEEEtran} 
\bibliography{Jalal_Ref}
\end{document}